\documentclass[11pt,twoside]{article}
\usepackage{amsmath, amsthm, amscd, amsfonts, amssymb, graphicx, color}

\date{}

\begin{document}

\centerline{}

\centerline {\Large{\bf $p$-frame relative to $b$-linear functional in $n$-Banach space}}

\centerline{}

\newcommand{\mvec}[1]{\mbox{\bfseries\itshape #1}}
\centerline{}
\centerline{\textbf{Prasenjit Ghosh}}
\centerline{Department of Pure Mathematics, University of Calcutta,}
\centerline{35, Ballygunge Circular Road, Kolkata, 700019, West Bengal, India}
\centerline{e-mail: prasenjitpuremath@gmail.com}
\centerline{}
\centerline{\textbf{T. K. Samanta}}
\centerline{Department of Mathematics, Uluberia College,}
\centerline{Uluberia, Howrah, 711315,  West Bengal, India}
\centerline{e-mail: mumpu$_{-}$tapas5@yahoo.co.in}

\newtheorem{Theorem}{\quad Theorem}[section]

\newtheorem{definition}[Theorem]{\quad Definition}

\newtheorem{theorem}[Theorem]{\quad Theorem}

\newtheorem{remark}[Theorem]{\quad Remark}

\newtheorem{corollary}[Theorem]{\quad Corollary}

\newtheorem{note}[Theorem]{\quad Note}

\newtheorem{lemma}[Theorem]{\quad Lemma}

\newtheorem{example}[Theorem]{\quad Example}

\newtheorem{result}[Theorem]{\quad Result}
\newtheorem{conclusion}[Theorem]{\quad Conclusion}

\newtheorem{proposition}[Theorem]{\quad Proposition}

\begin{abstract}
\textbf{\emph{Concept of $p$-frame with the help of $b$-linear functional in the case of $n$-Banach space is being presented and its few properties, one of them, Cartesian product of two \,$p$-frames again becomes a \,$p$-frame, have been discussed.\,Finally, the perturbation results and the stability of $p$-frame in $n$-Banach space with respect to $b$-linear functional are being studied.}}
\end{abstract}
{\bf Keywords:}  \emph{Frame, Banach frame, $p$-frame, $n$-Banach space, $b$-linear functional.}\\

{\bf 2020 Mathematics Subject Classification:} \emph{42C15; 46C07; 46M05; 47A80.}

\section{Introduction}

In mathematical research, frame theory is now a very effective field, as it has various applications in mathematics, science and engineering.\,In particular, frames are extensively used in image processing, sampling theory, signal processing, differential equation, geophysics, wireless sensor network and many more.\,Hilbert space is the standard setting for frame.\,Frame can also be generalized to the Banach space setting.\,The theoretical approach of frame for Banach space is quite different with respect to that of Hilbert space.\,Frame for Hilbert space was defined as a sequence of basis-like elements in Hilbert space.\,But, in Banach space, due to the absence of inner product, frame was completely defined as a sequence of bounded linear functionals from the dual space of the Banach space.\,Feichtinger and Groching \cite{F, K} extended the notion of frames to Banach spaces and presented the atomic decomposition for Banach spaces.\,Then Grochenig \cite{G} extended Banach frame in more general way in Banach space.\,Thereafter, further development of Banach frame was done by Casazza et al. \cite{H}.\,Banach frame allows elements of a Banach space to be written as a linear combination of the frame elements in a stable manner. 

Aldroubi et al. \cite{A} introduced \,$p$-frame in a Banach space and discussed some of its properties.\,Chistensen and stoeva \cite{C} also stuided \,$p$-frames in separable Banach spaces.\,Stoeva \cite{S} presented the reconstruction series of \,$p$-frame in a separable Banach spaces.

The notion of linear\;$2$-normed space was introduced by S.\,Gahler \cite{Gahler}.\;A survey of the theory of linear\;$2$-normed space can be found in \cite{Freese}.\;The concept of $2$-Banach space is briefly discussed in \cite{White}.\;H.\,Gunawan and Mashadi \cite{Mashadi} developed the generalization of a linear\;$2$-normed space for \,$n \,\geq\, 2$.\,P. Ghosh and T. K. Samanta \cite{Prasenjit, GP} studied the frames in \,$n$-Hilbert spaces and in their tensor products.

In this paper, we consider the \,$p$-frame relative to bounded\;$b$-linear functional in \,$n$-Banach space.\,We will see that the Cartesian products of two \,$p$-frames is also a \,$p$-frame in \,$n$-Banach space.\,A sufficient condition for the stability of \,$p$-frame for \,$n$-Banach space under some perturbations is discussed.\,We also present the finite sum of \,$p$-frame and establish a sufficient condition for the finite sum to be a \,$p$-frame in \,$n$-Banach space.\,Finally, a result related to the stability of finite sum of \,$p$-frames in \,$n$-Banach spaces is established.

Throughout this paper,\;$X_{1}$\, is considered to be a complex separable Banach space and \,$X_{1}^{\,\ast}$, its dual space.\,$\mathcal{B}\,(\,X_{1}\,)$\, denotes the space of all bounded linear operators on \,$X_{1}$.\,It is assumed that \,$p \,\in\, (\,1,\,\infty\,)$\, and when \,$p$\, and \,$q$\, are used in a same assertion, they satisfy the relation \,$1 \,/\, p \,+\, 1 \,/\, q \,=\, 1$.

\section{Preliminaries}

\smallskip\hspace{.6 cm} 

\begin{definition}\cite{O}
Let \,$\left(\,H,\, \left<\,\cdot,\,\cdot\,\right>\,\right)$\, be a Hilbert space.\,A sequence \,$\left\{\,f_{\,i}\,\right\}_{i \,=\, 1}^{\infty} \,\subseteq\, H$\, is said to be a frame for \,$H$\, if there exist positive constants \,$A,\, B$\, such that
\[ A\; \|\,f\,\|_{H}^{\,2} \,\leq\, \sum\limits_{i \,=\, 1}^{\infty}\, \left|\ \left <\,f \,,\, f_{\,i} \, \right >\,\right|^{\,2} \,\leq\, B \,\|\,f\,\|_{H}^{\,2}\; \;\forall\; f \,\in\, H.\]
The constants \,$A$\, and \,$B$\, are called frame bounds.
\end{definition}

\begin{definition}\cite{G}
Let \,$X_{1}$\, be a Banach space, \,$X_{d}$\, be a sequence space, which is a Banach space and for which the co-ordinate functionals are continuous.\,Let \,$\left\{\,g_{\,i}\,\right\}_{i \,\in\, I} \,\subseteq\, X_{1}^{\,\ast}$\, and \,$S  \,:\, X_{d} \,\to\, X_{1}$\, be a linear bounded operator.\,Then the pair \,$\left(\,\{\,g_{\,i}\,\},\,S\,\right)$\, is said to be a Banach frame for \,$X_{1}$\, with respect to \,$X_{d}$\, if   
\begin{description}
\item[$(i)$]$\left\{\,g_{\,i}\,(\,f\,)\,\right\} \,\in\, X_{\,d}\; \;\forall\; f \,\in\, X_{1}$,
\item[$(ii)$]there exist \,$B \,\geq\, A \,>\, 0$\, such that \,\[A\,\|\,f\,\|_{X_{1}} \,\leq\, \left\|\,\left\{\,g_{\,i}\,(\,f\,)\,\right\}\,\right\|_{X_{d}} \,\leq\, B\,\|\,f\,\|_{X_{1}}\;\;\forall\, f \,\in\, X_{1},\]
\item[$(iii)$]$S\,\left(\,\left\{\,g_{\,i}\,(\,f\,)\,\right\}\,\right) \,=\, f\; \;\forall\; f \,\in\, X_{1}$.
\end{description}
The constants \,$A,\,B$\, are called Banach frame bounds and \,$S$\, is called the reconstruction operator.
\end{definition}

For particular \,$X_{d} \,=\, l^{\,p}$, Aldroubi et al. \cite{A} introduced the notion of \,$p$-frame.

\begin{definition}\cite{A}
Let \,$1 \,<\, p \,<\, \infty$.\,A countable family \,$\left\{\,g_{\,i}\,\right\}_{i \,\in\, I} \,\subseteq\, X_{1}^{\,\ast}$\, is said to be a p-frame for \,$X_{1}$\, if there exist constants \,$0 \,<\, A \,\leq\, B \,<\, \infty$\, such that
\begin{equation}\label{eq.1}
A \,\left\|\,f\,\right\|_{X_{1}} \,\leq\, \left(\,\sum\limits_{\,i \,\in\, I}\,\left|\,g_{\,i}\,(\,f\,)\,\right|^{\,p}\,\right)^{\,1 \,/\, p} \,\leq\, B \,\left\|\, f\,\right\|_{X_{1}}\; \;\forall\; f \,\in\, X_{1}.
\end{equation}
If the family \,$\left\{\,g_{\,i}\,\right\}_{i \,\in\, I}$\,  satisfies only the right inequality of (\ref{eq.1}), it is called a \,$p$-Bessel sequence in \,$X_{1}$\, with bound \,$B$.  
\end{definition}

\begin{definition}\cite{C}
Let \,$\left\{\,g_{\,i}\,\right\}_{i \,\in\, I} \,\subseteq\, X_{1}^{\,\ast}$\, be a p-frame for \,$X_{1}$.\,Then the operator \,$U \,:\, X_{1} \,\to\, l^{\,p}$\, defined by \,$U\,f \,=\, \{\,g_{\,i}\,(\,f\,)\,\}_{i \,\in\, I}$\, is called the analysis operator and the operator given by \,$T \,:\, l^{\,q} \,\to\, X_{1}^{\,\ast}, \;T\,\{\,d_{\,i}\,\} \,=\, \sum\limits_{i \,\in\, I}\,d_{\,i}\,g_{\,i}$\, is called the synthesis operator. 
\end{definition}

\begin{definition}\cite{Mashadi}
A \,$n$-norm on a linear space \,$X$\, (\,over the field \,$\mathbb{K}$\, of real or complex numbers\,) is a function
\[\left(\,x_{\,1},\, x_{\,2},\, \cdots,\, x_{\,n}\,\right) \,\longmapsto\, \left\|\,x_{\,1},\, x_{\,2},\, \cdots,\, x_{\,n}\,\right\|,\; x_{\,1},\, x_{\,2},\, \cdots,\, x_{\,n} \,\in\, X,\]from \,$X^{\,n}$\, to the set \,$\mathbb{R}$\, of all real numbers such that for every \,$x_{\,1},\, x_{\,2},\, \cdots,\, x_{\,n} \,\in\, X$\, and \,$\alpha \,\in\, \mathbb{K}$,
\begin{description}
\item[$(i)$]\;\; $\left\|\,x_{\,1},\, x_{\,2},\, \cdots,\, x_{\,n}\,\right\| \,=\, 0$\; if and only if \,$x_{\,1},\, \cdots,\, x_{\,n}$\; are linearly dependent,
\item[$(ii)$]\;\;\; $\left\|\,x_{\,1},\, x_{\,2},\, \cdots,\, x_{\,n}\,\right\|$\; is invariant under permutations of \,$x_{\,1},\, x_{\,2},\, \cdots,\, x_{\,n}$,
\item[$(iii)$]\;\;\; $\left\|\,\alpha\,x_{\,1},\, x_{\,2},\, \cdots,\, x_{\,n}\,\right\| \,=\, |\,\alpha\,|\, \left\|\,x_{\,1},\, x_{\,2},\, \cdots,\, x_{\,n}\,\right\|$,
\item[$(iv)$]\;\; $\left\|\,x \,+\, y,\, x_{\,2},\, \cdots,\, x_{\,n}\,\right\| \,\leq\, \left\|\,x,\, x_{\,2},\, \cdots,\, x_{\,n}\,\right\| \,+\,  \left\|\,y,\, x_{\,2},\, \cdots,\, x_{\,n}\,\right\|$.
\end{description}
A linear space \,$X$, together with a n-norm \,$\left\|\,\cdot,\, \cdots,\, \cdot \,\right\|$, is called a linear n-normed space. 
\end{definition}

\begin{definition}\cite{Mashadi}
A sequence \,$\{\,x_{\,k}\,\} \,\subseteq\, X$\, is said to converge to \,$x \,\in\, X$\; if 
\[\lim\limits_{k \to \infty}\,\left\|\,x_{\,k} \,-\, x,\, x_{\,2},\, \cdots,\, x_{\,n} \,\right\| \,=\, 0\; \;\forall\; x_{\,2},\, \cdots,\, x_{\,n} \,\in\, X\]
and it is called a Cauchy sequence if 
\[\lim\limits_{l,\, k \to \infty}\,\left\|\,x_{\,l} \,-\, x_{\,k},\, x_{\,2},\, \cdots,\, x_{\,n}\,\right\| \,=\, 0\; \;\forall\; x_{\,2},\, \cdots,\, x_{\,n} \,\in\, X.\]
The space \,$X$\, is said to be complete or n-Banach space if every Cauchy sequence in this space is convergent in \,$X$.
\end{definition}

\begin{theorem}\cite{Lahiri}
If \,$a$\, and \,$b$\, are real or complex numbers and \,$p \,\geq\, 1$, then 
\begin{equation}\label{eqq1}
\left|\,a \,+\, b\,\right|^{\,p} \,\leq\, 2^{\,p}\,\left(\,\left|\,a\,\right|^{\,p} \,+\, \left|\,b\,\right|^{\,p}\,\right).
\end{equation}
\end{theorem}

\section{$p$-frame in $n$-Banach space}

\smallskip\hspace{.6 cm} In this section, we first define a bounded\;$b$-linear functional on \,$W \,\times\,\left<\,a_{\,2}\,\right> \,\times\, \cdots \,\times\, \left<\,a_{\,n}\,\right>$, where \,$W$\, be a subspace of \,$X$\, and \,$a_{\,2},\,\cdots,\, a_{\,n} \,\in\, X$\, and then the notion of \,$p$-frame in \,$n$-Banach space \,$X$\, is discussed.  

\begin{definition}\label{def1}
Let \,$\left(\,X,\, \|\,\cdot,\, \cdots,\, \cdot\,\|_{X} \,\right)$\, be a linear n-normed space and \,$a_{\,2},\, \cdots,\, a_{\,n}$\, be fixed elements in \,$X$.\;Let \,$W$\, be a subspace of \,$X$\, and \,$\left<\,a_{\,i}\,\right>$\, denote the subspaces of \,$X$\, generated by \,$a_{\,i}$, for \,$i \,=\, 2,\, 3,\, \cdots,\,n $.\;Then a map \,$T \,:\, W \,\times\,\left<\,a_{\,2}\,\right> \,\times\, \cdots \,\times\, \left<\,a_{\,n}\,\right> \,\to\, \mathbb{K}$\; is called a b-linear functional on \,$W \,\times\, \left<\,a_{\,2}\,\right> \,\times\, \cdots \,\times\, \left<\,a_{\,n}\,\right>$, if for every \,$x,\, y \,\in\, W$\, and \,$k \,\in\, \mathbb{K}$, the following conditions hold:
\begin{description}
\item[$(i)$] $T\,(\,x \,+\, y,\, a_{\,2},\, \cdots,\, a_{\,n}\,) \,=\, T\,(\,x,\, a_{\,2},\, \cdots,\, a_{\,n}\,) \,+\, T\,(\,y,\, a_{\,2},\, \cdots,\, a_{\,n}\,)$
\item[$(ii)$] $T\,(\,k\,x,\, a_{\,2},\, \cdots,\, a_{\,n}\,) \,=\, k\; T\,(\,x,\, a_{\,2},\, \cdots,\, a_{\,n}\,)$. 
\end{description}
A b-linear functional is said to be bounded if \,$\exists$\, a real number \,$M \,>\, 0$\; such that
\[\left|\,T\,(\,x,\, a_{\,2},\, \cdots,\, a_{\,n}\,)\,\right| \,\leq\, M\; \left\|\,x,\, a_{\,2},\, \cdots,\, a_{\,n}\,\right\|_{X}\; \;\forall\; x \,\in\, W.\]
The norm of the bounded b-linear functional \,$T$\, is defined by
\[\|\,T\,\| \,=\, \inf\,\left\{\,M \,>\, 0 \,:\, \left|\,T\,(\,x,\, a_{\,2},\, \cdots,\, a_{\,n}\,)\,\right| \,\leq\, M\, \left\|\,x,\, a_{\,2},\, \cdots,\, a_{\,n}\,\right\|_{X} \;\forall\; x \,\in\, W\,\right\}.\]
The norm of \,$T$\, can be expressed by any one of the following equivalent formula:
\begin{description}
\item[$(i)$]$\|\,T\,\| \,=\, \sup\,\left\{\,\left|\,T\,(\,x,\, a_{\,2},\, \cdots,\, a_{\,n}\,)\,\right| \;:\; \left\|\,x,\, a_{\,2},\, \cdots,\, a_{\,n}\,\right\|_{X} \,\leq\, 1\,\right\}$.
\item[$(ii)$]$\|\,T\,\| \,=\, \sup\,\left\{\,\left|\,T\,(\,x,\, a_{\,2},\, \cdots,\, a_{\,n}\,)\,\right| \;:\; \left\|\,x,\, a_{\,2},\, \cdots,\, a_{\,n}\,\right\|_{X} \,=\, 1\,\right\}$.
\item[$(iii)$]$ \|\,T\,\| \,=\, \sup\,\left \{\,\dfrac{\left|\,T\,(\,x,\, a_{\,2},\, \cdots,\, a_{\,n}\,)\,\right|}{\left\|\,x,\, a_{\,2},\, \cdots,\, a_{\,n}\,\right\|_{X}} \;:\; \left\|\,x,\, a_{\,2},\, \cdots,\, a_{\,n}\,\right\|_{X} \,\neq\, 0\,\right \}$. 
\end{description}
\end{definition}

Let \,$X_{F}^{\,\ast}$\, be the Banach space of all bounded\;$b$-linear functional defined on \,$X \,\times\, \left<\,a_{\,2}\,\right> \,\times \cdots \,\times\, \left<\,a_{\,n}\,\right>$.\,It is easy to verify that
\[\left\|\,x,\, a_{\,2},\, \cdots,\, a_{\,n}\,\right\|_{X} \,=\, \sup\,\left\{\, \dfrac{\left|\,T\,(\,x,\, a_{\,2},\, \cdots,\, a_{\,n}\,)\,\right|}{\|\,T\,\|} \,:\, T \,\in\, X^{\,\ast}_{F} \,,\; T \,\neq\, 0 \,\right\}.\]
Some properties of bounded\;$b$-linear functional defined on \,$X \,\times\, \left<\,a_{\,2}\,\right> \,\times\, \cdots \,\times\, \left<\,a_{\,n}\,\right>$\, have been discussed in \cite{TK}.\,For the remaining part of this paper, \,$X$\, denotes the \,$n$-Banach space with respect to the \,$n$-norm \,$\|\,\cdot,\, \cdots,\, \cdot\,\|_{X}$.

\begin{definition}
Let \,$1 \,<\, p \,<\, \infty$\, and \,$a_{\,2},\, \cdots,\, a_{\,n}$\, be fixed elements in \,$X$.\,A countable family \,$\left\{\,T_{i}\,\right\}_{i \,\in\, I} \,\subseteq\, X_{F}^{\,\ast}$\, is called a p-frame associated to \,$\left(\,a_{\,2},\, \cdots,\, a_{\,n}\,\right)$\, for \,$X$\, if there exist constants \,$A,\,B \,>\, 0$\, such that
\begin{equation}\label{eq1}
A \,\left\|\,x,\, a_{\,2},\, \cdots,\, a_{\,n}\,\right\|^{\,p}_{X} \,\leq\, \sum\limits_{\,i \,\in\, I}\,\left|\,T_{\,i}\,(\,x,\, a_{\,2},\, \cdots,\, a_{\,n}\,)\,\right|^{\,p} \,\leq\, B\,\left\|\,x,\, a_{\,2},\, \cdots,\, a_{\,n}\,\right\|^{\,p}_{X}
\end{equation}
for all \,$x \,\in\, X$.\,The constants \,$A$\, and \,$B$\, are called the lower and upper frame bounds.\,A p-frame associated to \,$\left(\,a_{\,2},\, \cdots,\, a_{\,n}\,\right)$\, is said to be a tight if \,$A \,=\, B$.\,It is called Parseval p-frame associated to \,$\left(\,a_{\,2},\, \cdots,\, a_{\,n}\,\right)$\, if \,$A \,=\, B \,=\, 1$.\,If the family \,$\left\{\,T_{i}\,\right\}_{i \,\in\, I}$\, satisfies only right inequality of (\ref{eq1}), it is called a p-Bessel sequence associated to \,$\left(\,a_{\,2},\, \cdots,\, a_{\,n}\,\right)$\, in \,$X$\, with bound \,$B$.   
\end{definition}

\begin{remark}
Suppose that \,$\left\{\,T_{\,i}\,\right\}_{i \,\in\, I}$\, is a tight p-frame associated to \,$\left(\,a_{\,2},\, \cdots,\, a_{\,n}\,\right)$\, for \,$X$\, with bound \,$A$.\,Then for all \,$x \,\in\, X$, we have
\[\sum\limits_{\,i \,\in\, I}\,\left|\,T_{\,i}\,(\,x,\, a_{\,2},\, \cdots,\, a_{\,n}\,)\,\right|^{\,p} \,=\, A\,\left\|\,x,\, a_{\,2},\, \cdots,\, a_{\,n}\,\right\|^{\,p}_{X}\]
\[\Rightarrow\,\sum\limits_{\,i \,\in\, I}\,\left|\,A^{\,-\,1 \,/\, p}\;T_{\,i}\,(\,x,\, a_{\,2},\, \cdots,\, a_{\,n}\,)\,\right|^{\,p} \,=\, \left\|\,x,\, a_{\,2},\, \cdots,\, a_{\,n}\,\right\|^{\,p}_{X}\]
This verify that \,$\left\{\,A^{\,-\, 1 \,/\, p}\;T_{\,i}\,\right\}_{i \,\in\, I}$\, is a Parseval p-frame associated to \,$\left(\,a_{\,2},\, \cdots,\, a_{\,n}\,\right)$\, for \,$X$.   
\end{remark}

\begin{theorem}
Let \,$\left\{\,T_{\,i}\,\right\}_{i \,\in\, I}$\, and \,$\left\{\,U_{\,i}\,\right\}_{i \,\in\, I}$\, be the p-Bessel sequences associated to \,$\left(\,a_{\,2},\, \cdots,\, a_{\,n}\,\right)$\, in \,$X$\, with bounds \,$A_{\,1}$\, and $A_{\,2}$, respectively.\,Then \,$\left\{\,T_{\,i} \,+\, U_{\,i}\,\right\}_{i \,\in\, I}$\, is a p-Bessel sequence associated to \,$\left(\,a_{\,2},\, \cdots,\, a_{\,n}\,\right)$\, in \,$X$\, with bound \,$2^{\,p}\,\max\left(\,A_{\,1},\,A_{\,2}\,\right)$. 
\end{theorem}

\begin{proof}
For each \,$x \,\in\, X$, we have 
\[\sum\limits_{\,i \,\in\, I}\left|\left(\,T_{\,i} + U_{\,i}\right)(\,x,\, a_{\,2},\, \cdots,\, a_{\,n}\,)\right|^{p} = \sum\limits_{\,i \,\in\, I}\left|\,T_{\,i}\,(\,x,\, a_{\,2},\, \cdots,\, a_{\,n}\,) + U_{\,i}\,(\,x,\, a_{\,2},\, \cdots,\, a_{\,n}\,)\,\right|^{p}\]
\[\leq\, 2^{\,p}\,\left(\,\sum\limits_{\,i \,\in\, I}\,\left|\,T_{\,i}\,(\,x,\, a_{\,2},\, \cdots,\, a_{\,n}\,)\,\right|^{\,p} \,+\, \sum\limits_{\,i \,\in\, I}\,\left|\,U_{\,i}\,(\,x,\, a_{\,2},\, \cdots,\, a_{\,n}\,)\,\right|^{\,p}\,\right)\;[\;\text{by $(\ref{eqq1})$}\;]\]
\[\leq\, 2^{\,p}\,\left(\,A_{\,1}\,\left\|\,x,\, a_{\,2},\, \cdots,\, a_{\,n}\,\right\|^{\,p}_{X}\,+\, A_{\,2}\,\left\|\,x,\, a_{\,2},\, \cdots,\, a_{\,n}\,\right\|^{\,p}_{X}\,\right)\hspace{3.5cm}\]
\[\leq\, 2^{\,p}\,\max\left(\,A_{\,1},\,A_{\,2}\,\right)\,\left\|\,x,\, a_{\,2},\, \cdots,\, a_{\,n}\,\right\|^{\,p}_{X}.\hspace{6.1cm}\]
Thus \,$\left\{\,T_{\,i} \,+\, U_{\,i}\,\right\}_{i \,\in\, I}$\, is a \,$p$-Bessel sequence associated to \,$\left(\,a_{\,2},\, \cdots,\, a_{\,n}\,\right)$\, in \,$X$\, with bound \,$2^{\,p}\,\max\left(\,A_{\,1},\,A_{\,2}\,\right)$.
\end{proof}

\begin{definition}
A sequence \,$\left\{\,f_{\,i}\,\right\}_{i \,\in\, I} \,\subseteq\, X$\, is said to be a \,$q$-frame associated to \,$\left(\,a_{\,2},\, \cdots,\, a_{\,n}\,\right)$\, for \,$X_{F}^{\,\ast}$\, if there exist constants \,$A,\,B \,>\, 0$\, such that
\begin{equation}\label{epqr1}
A \,\left\|\,T\,\right\|^{\,q}_{X^{\,\ast}_{F}} \,\leq\, \sum\limits_{\,i \,\in\, I}\,\left|\,T\,(\,f_{\,i},\, a_{\,2},\, \cdots,\, a_{\,n}\,)\,\right|^{\,q} \,\leq\, B\,\left\|\,T\,\right\|^{\,q}_{X^{\,\ast}_{F}}\; \;\forall\; T \,\in\, X_{F}^{\,\ast}.
\end{equation}
If the family \,$\left\{\,f_{\,i}\,\right\}_{i \,\in\, I}$\, satisfies only the right inequality of (\ref{epqr1}), it is called a q-Bessel sequence associated to \,$\left(\,a_{\,2},\, \cdots,\, a_{\,n}\,\right)$\, for \,$X^{\,\ast}_{F}$.
\end{definition}

\begin{lemma}
Suppose that \,$\left\{\,T_{i}\,\right\}_{i \,\in\, I} \,\subseteq\, X_{F}^{\,\ast}$\, satisfies the lower p-frame condition.\,Then the operator \,$U \,:\, \mathcal{D}\,(\,U\,) \,\subseteq\, X \,\to\, l^{\,p},\; U\,x \,=\, \left\{\,T_{\,i}\,(\,x,\,a_{\,2},\, \cdots,\, a_{\,n}\,)\,\right\}_{i \,\in\, I}$, where \,$\,\mathcal{D}\,(\,U\,) \,=\, \left\{\,x \,\in\, X \,:\, \sum\limits_{\,i \,\in\, I}\,\left|\,T_{\,i}\,(\,x,\, a_{\,2},\, \cdots,\, a_{\,n}\,)\,\right|^{\,p} \,<\, \infty\,\right\}$, is a closed operator.  
\end{lemma}

\begin{proof}
To prove that \,$U$\, is closed, consider a sequence \,$\left\{\,x_{\,k}\,\right\} \,\subset\, \mathcal{D}\,(\,U\,)$\, for which
\[x_{\,k} \,\to\, x\; \;\text{in}\; X\; \;\text{and}\; \;U\,x_{\,k} \,\to\, \left\{\,c_{\,i}\,\right\}_{i \,\in\, I}\; \;\text{in}\; \;l^{\,p}\; \;\text{when}\; k \,\to\, \infty.\]
Since all \,$T_{\,i}$\, are bounded\;$b$-linear functionals, for all \,$i$, we have
\[\left|\,T_{\,i}\,\left(\,x_{\,k},\, a_{\,2},\, \cdots,\, a_{\,n}\,\right) \,-\, T_{\,i}\,(\,x,\, a_{\,2},\, \cdots,\, a_{\,n}\,)\,\right| \,=\, \left|\,T_{\,i}\,\left(\,x_{\,k} \,-\, x,\, a_{\,2},\, \cdots,\, a_{\,n}\,\right)\,\right|\] 
\[\hspace{1cm}\leq\, \left\|\,T_{\,i}\,\right\|\,\left\|\,x_{\,k} \,-\, x,\, a_{\,2},\, \cdots,\, a_{\,n}\,\right\|_{X} \,\to\, 0\; \;\text{as}\; \;k \,\to\, \infty.\]
\[\text{That is},\hspace{.3cm} T_{\,i}\,\left(\,x_{\,k},\, a_{\,2},\, \cdots,\, a_{\,n}\,\right) \,\to\, T_{\,i}\,(\,x,\, a_{\,2},\, \cdots,\, a_{\,n}\,)\; \;\text{as}\; \;k \,\to\, \infty,\; \,\forall\; i.\hspace{3cm}\]
Now, using the assumption and definition of \,$U$, we get \,$T_{\,i}\,\left(\,x_{\,k},\, a_{\,2},\, \cdots,\, a_{\,n}\,\right) \,\to\, c_{\,i}$\, as \,$k \,\to\, \infty$.\,Thus \,$\left\{\,T_{\,i}\,(\,x,\,a_{\,2},\, \cdots,\, a_{\,n}\,)\,\right\}_{i \,\in\, I} \,=\, \left\{\,c_{\,i}\,\right\}_{i \,\in\, I}$.\,This shows that \,$x \,\in\, \mathcal{D}\,(\,U\,)$\, and \,$U\,x \,=\, \left\{\,c_{\,i}\,\right\}_{i \,\in\, I}$, and hence \,$U$\, is closed.
\end{proof}

\begin{lemma}
Let \,$\left\{\,T_{i}\,\right\}_{i \,\in\, I} \,\subseteq\, X_{F}^{\,\ast}$\, and assume that there exists a q-Bessel sequence \,$\left\{\,f_{\,i}\,\right\}_{i \,\in\, I} \,\subseteq\, X$\,  associated to \,$\left(\,a_{\,2},\, \cdots,\, a_{\,n}\,\right)$\, for \,$X_{F}^{\,\ast}$\, such that 
\[x \,=\, \sum\limits_{\,i \,\in\, I}\,T_{\,i}\,(\,x,\, a_{\,2},\, \cdots,\, a_{\,n}\,)\,f_{\,i},\; \text{for all \,$x \,\in\, \mathcal{D}\,(\,U\,)$}.\]
Then the sequence \,$\left\{\,T_{i}\,\right\}_{i \,\in\, I}$\, satisfies the lower \,$p$-frame condition for \,$X$.  
\end{lemma}

\begin{proof}
Assume that \,$\left\{\,f_{\,i}\,\right\}_{i \,\in\, I}$\, satisfies the upper \,$q$-frame condition for \,$X_{F}^{\,\ast}$\, with a bound \,$B$.\,For every \,$x \,\in\, \mathcal{D}\,(\,U\,)$, the sequence \,$\left\{\,T_{\,i}\,(\,x,\, a_{\,2},\, \cdots,\, a_{\,n}\,)\,\right\}_{i \,\in\, I}$\, belongs to \,$l^{\,p}$\, and
\[\left\|\,x,\,a_{\,2},\,\cdots,\,a_{\,n}\,\right\|_{X} \,=\, \sup\limits_{R \,\in\, X_{F}^{\,\ast},\,\|\,R\,\| \,\leq\, 1}\,\left|\,R\,\left(\,\sum\limits_{\,i \,\in\, I}\,T_{\,i}\,(\,x,\, a_{\,2},\, \cdots,\, a_{\,n}\,)\,f_{\,i},\,a_{\,2},\,\cdots,\,a_{\,n}\,\right)\,\right|\]
\[\leq\, \sup\limits_{R \,\in\, X_{F}^{\,\ast},\,\|\,R\,\| \,\leq\, 1}\,\left\|\,\left\{\,T_{\,i}\,(\,x,\, a_{\,2},\, \cdots,\, a_{\,n}\,)\,\right\}_{i \,\in\, I}\,\right\|_{\,p}\,\left\|\,\left\{\,R\,(\,f_{\,i},\, a_{\,2},\, \cdots,\, a_{\,n}\,)\,\right\}_{i \,\in\, I}\,\right\|_{\,q}\hspace{1.5cm}\]
\[\leq\, B^{\,1 \,/\, q}\,\left\|\,\left\{\,T_{\,i}\,(\,x,\, a_{\,2},\, \cdots,\, a_{\,n}\,)\,\right\}_{i \,\in\, I}\,\right\|_{\,p} \,=\, B^{\,1 \,/\, q}\,\left(\,\sum\limits_{\,i \,\in\, I}\,\left|\,T_{\,i}\,(\,x,\, a_{\,2},\, \cdots,\, a_{\,n}\,)\,\right|^{\,p}\,\right)^{\,1 \,/\, p}\]
\[\Rightarrow\, B^{\,-\, p \,/\, q}\,\left\|\,x,\,a_{\,2},\,\cdots,\,a_{\,n}\,\right\|^{\,p}_{X} \,\leq\, \sum\limits_{\,i \,\in\, I}\,\left|\,T_{\,i}\,(\,x,\, a_{\,2},\, \cdots,\, a_{\,n}\,)\,\right|^{\,p}.\hspace{3.7cm}\]
Obviously, for \,$x \,\in\, X \,\setminus\, \mathcal{D}\,(\,U\,)$, the lower \,$p$-frame condition is satisfied. 
\end{proof}

\begin{theorem}
Let \,$\left\{\,T_{i}\,\right\}_{i \,\in\, I} \,\subseteq\, X_{F}^{\,\ast}$\, be a \,$p$-Bessel sequence associated to \,$(\,a_{\,2},\, \cdots$, \,$a_{\,n}\,)$\, in \,$X$\, with bound \,$B$\, and \,$\left\{\,f_{\,i}\,\right\}_{i \,\in\, I} \,\subseteq\, X$\, be a \,$q$-Bessel sequence associated to \,$\left(\,a_{\,2},\, \cdots,\, a_{\,n}\,\right)$\, for \,$X_{F}^{\,\ast}$\, such that 
\[x \,=\, \sum\limits_{\,i \,\in\, I}\,T_{\,i}\,(\,x,\, a_{\,2},\, \cdots,\, a_{\,n}\,)\,f_{\,i}, \;\text{for all}\; x \,\in\, X.\]
Then following statements are hold:
\begin{description}
\item[$(i)$]\,$\left\{\,f_{\,i}\,\right\}_{i \,\in\, I}$\, is a \,$q$-frame associated to \,$\left(\,a_{\,2},\, \cdots,\, a_{\,n}\,\right)$\, for \,$X_{F}^{\,\ast}$\, with lower bound \,$B^{\,-\, q \,/\, p}$.
\item[$(ii)$]\,$R \,=\, \sum\limits_{\,i \,\in\, I}\,R\,\left(\,f_{\,i},\, a_{\,2},\, \cdots,\, a_{\,n}\,\right)\,T_{\,i}\; \;\forall\; R \,\in\, X_{F}^{\,\ast}$.  
\end{description} 
\end{theorem}

\begin{proof}$(i)$
Since \,$\left\{\,f_{\,i}\,\right\}_{i \,\in\, I}$\, is a \,$q$-Bessel sequence  associated to \,$\left(\,a_{\,2},\, \cdots,\, a_{\,n}\,\right)$\, for \,$X_{F}^{\,\ast}$, for every \,$R \,\in\, X_{F}^{\,\ast}$, the sequence \,$\left\{\,R\,(\,x,\, a_{\,2},\, \cdots,\, a_{\,n}\,)\,\right\}_{i \,\in\, I}$\, belongs to \,$l^{\,q}$.\,Now,
\[\|\,R\,\|_{X_{F}^{\,\ast}} \,=\, \sup\limits_{x \,\in\, X,\;\left\|\,x,\,a_{\,2},\,\cdots,\,a_{\,n}\,\right\|_{X} \,=\, 1}\,\left|\,R\,(\,x,\, a_{\,2},\, \cdots,\, a_{\,n}\,)\,\right|\hspace{4.7cm}\]
\[=\, \sup\limits_{x \,\in\, X,\;\left\|\,x,\,a_{\,2},\,\cdots,\,a_{\,n}\,\right\|_{X} \,=\, 1}\,\left|\,\sum\limits_{\,i \,\in\, I}\,R\,\left(\,f_{\,i},\, a_{\,2},\, \cdots,\, a_{\,n}\,\right)\,T_{\,i}\,(\,x,\, a_{\,2},\, \cdots,\, a_{\,n}\,)\,\right|\hspace{2.5cm}\]
\[\leq\, \sup\limits_{x \,\in\, X,\,\left\|\,x,\,a_{\,2},\,\cdots,\,a_{\,n}\,\right\|_{X} \,=\, 1}\,\left\|\,\left\{\,T_{\,i}\,(\,x,\, a_{\,2},\, \cdots,\, a_{\,n}\,)\,\right\}_{i \,\in\, I}\,\right\|_{p}\,\left\|\,\left\{\,R\,(\,f_{\,i},\, a_{\,2},\, \cdots,\, a_{\,n}\,)\,\right\}_{i \,\in\, I}\,\right\|_{q}\]
\[\leq\, B^{\,1 \,/\, p}\,\left\|\,\left\{\,R\,(\,f_{\,i},\, a_{\,2},\, \cdots,\, a_{\,n}\,)\,\right\}_{i \,\in\, I}\,\right\|_{\,q} \,=\, B^{\,1 \,/\, p}\,\left(\,\sum\limits_{\,i \,\in\, I}\,\left|\,R\,(\,f_{\,i},\, a_{\,2},\, \cdots,\, a_{\,n}\,)\,\right|^{\,q}\,\right)^{\,1 \,/\, q}.\]
\[\bigg[\,\text{since \,$\left\{\,T_{i}\,\right\}_{i \,\in\, I}$\, is a \,$p$-Bessel sequence associated to \,$\left(\,a_{\,2},\, \cdots,\, a_{\,n}\,\right)$\, with bound \,$B$}\,\bigg].\]
This shows that \,$\left\{\,f_{\,i}\,\right\}_{i \,\in\, I}$\, is a \,$q$-frame associated to \,$\left(\,a_{\,2},\, \cdots,\, a_{\,n}\,\right)$\, for \,$X_{F}^{\,\ast}$\, with lower bound $B^{\,-\, q \,/\, p}$.\\

$(ii)$\;\;For the proof of \,$(ii)$, we use \,$\mathbb{N}$\, as an index set.\,Fix an arbitrary \,$R \,\in\, X_{F}^{\,\ast}$.\,By the above similar calculation, for \,$k \,\in\, \mathbb{N}$, we get
\[\left\|\,R \,-\, \sum\limits_{i \,=\, 1}^{\,k}\,R\,\left(\,f_{\,i},\, a_{\,2},\, \cdots,\, a_{\,n}\,\right)\,T_{\,i}\,\right\|_{X_{F}^{\,\ast}}\hspace{2cm}\]
\[ \,\leq\, B^{\,1 \,/\, p}\, \left(\,\sum\limits_{i \,=\, k \,+\, 1}\,\left|\,R\,\left(\,f_{\,i},\, a_{\,2},\, \cdots,\, a_{\,n}\,\right)\,\right|^{\,q}\,\right)^{\,1 \,/\, q}\to\, 0\;\; \;\text{as}\; k \,\to\, \infty.\]
This completes the proof.  
\end{proof}

\begin{theorem}
The sequence \,$\left\{\,T_{i}\,\right\}_{i \,\in\, I} \,\subseteq\, X_{F}^{\,\ast}$\, is a \,$p$-Bessel sequence associated to \,$\left(\,a_{\,2},\, \cdots,\, a_{\,n}\,\right)$\, in \,$X$\, with bound \,$B$\, if and only if the operator given by \[T \,\,:\, l^{\,q} \,\to\, X_{F}^{\,\ast},\; \;T\,\left(\,\left\{\,d_{\,i}\,\right\}\,\right) \,=\, \sum\limits_{\,i \,\in\, I}\,d_{\,i}\,T_{\,i},\; \;\text{for all}\; \left\{\,d_{\,i}\,\right\} \,\in\, l^{\,q}\] is well-defined and bounded operator and \,$\|\,T\,\|^{\,p} \,\leq\, B$. 
\end{theorem}

\begin{proof}
First we consider that \,$\left\{\,T_{\,i}\,\right\}_{i \,\in\, I} \,\subseteq\, X_{F}^{\,\ast}$\, is a \,$p$-Bessel sequence associated to \,$\left(\,a_{\,2},\, \cdots,\, a_{\,n}\,\right)$\, in \,$X$\, with bound \,$B$.\,For \,$k \,>\, l$, we have
\[\left\|\,\sum\limits_{\,i \,=\, 1}^{\,l}\,d_{\,i}\,T_{\,i} \,-\, \sum\limits_{\,i \,=\, 1}^{\,k}\,d_{\,i}\,T_{\,i}\,\right\| \,=\, \sup\limits_{x \,\in\, X,\;\left\|\,x,\, a_{\,2},\, \cdots,\, a_{\,n}\,\right\|_{X} \,=\, 1}\,\left|\,\sum\limits_{\,i \,=\, k\,+\,1}^{\,l}\,d_{\,i}\,T_{\,i}\,(\,x,\,a_{\,2},\, \cdots,\, a_{\,n}\,)\,\right|\] 
\[\leq\, \sup\limits_{x \,\in\, X,\;\left\|\,x,\, a_{\,2},\, \cdots,\, a_{\,n}\,\right\|_{X} \,=\, 1}\,\left(\,\sum\limits_{\,i \,=\, k\,+\,1}^{\,l}\,|\,d_{\,i}\,|^{\,q}\,\right)^{\,1 \,/\, q}\,\left(\,\sum\limits_{\,i \,=\, k\,+\,1}^{\,l}\,\left|\,T_{\,i}\,(\,x,\, a_{\,2},\, \cdots,\, a_{\,n}\,)\,\right|^{\,p}\,\right)^{\,1 \,/\, p}\hspace{1cm}\]
\[\leq\, B^{\,1 \,/\, p}\,\left(\,\sum\limits_{\,i \,=\, k\,+\,1}^{\,l}\,|\,d_{\,i}\,|^{\,q}\,\right)^{\, 1 \,/\, q}\; \bigg[\;\text{$\left\{\,T_{i}\,\right\}$\, is \,$p$-Bessel sequence associated to \,$\left(\,a_{\,2},\, \cdots,\, a_{\,n}\,\right)$}\,\bigg]\]
Therefore, \,$\sum\limits_{\,i \,\in\, I}\,d_{\,i}\,T_{\,i}$\, is convergent and \,$T$\, is well-defined.\,From the above calculation also it follows that \,$\left\|\,T\,\{\,d_{\,i}\,\}\,\right\| \,\leq\, B^{\,1 \,/\, p}\,\left(\,\sum\limits_{\,i \,\in\, I}\,|\,d_{\,i}\,|^{\,q}\,\right)^{\, 1 \,/\, q}$, i\,.\,e.,\;$T$\, is bounded and \,$\|\,T\,\|^{\,p} \,\leq\, B$.\\

Conversely, suppose that \,$T$\, is well-defined and bounded with \,$\|\,T\,\|^{\,p} \,\leq\, B$.\,For given \,$x \,\in X$, the map 
\[\phi_{\,x} \,:\, \{\,d_{\,i}\,\} \,\to\, T\,\left(\,\{\,d_{\,i}\,\}\,\right)\,(\,x,\,a_{\,2},\, \cdots,\, a_{\,n}\,) \,=\, \sum\limits_{\,i \,\in\, I}\,d_{\,i}\,T_{\,i}\,(\,x,\,a_{\,2},\, \cdots,\, a_{\,n}\,)\]
is a continuous linear functional on \,$l^{\,q}$, so \,$\left\{\,T_{\,i}\,(\,x,\,a_{\,2},\, \cdots,\, a_{\,n}\,)\,\right\}_{i \,\in\, I} \,\in\, l^{\,p}$.\,Let \,$\{\,e_{\,i}\,\}$\, be the natural basis for \,$l^{\,q}$.\,Then we have
\[\left\|\,\phi_{\,x}\,\right\| \,=\, \left(\,\sum\limits_{\,i \,\in\, I}\,\left|\,\phi_{\,x}\,(\,e_{\,i}\,)\,\right|^{\,p}\,\right)^{\,1 \,/\, p} \,=\, \left(\,\sum\limits_{\,i \,\in\, I}\,\left|\,T_{\,i}\,(\,x,\,a_{\,2},\,\cdots,\,a_{\,n}\,)\,\right|^{\,p}\,\right)^{\,1 \,/\, p}.\]
Now, for all \,$\{\,d_{\,i}\,\} \,\in\, l^{\,q}$, we have
\[\left|\,\phi_{\,x}\,\{\,d_{\,i}\,\}\,\right| = \left|\,T\,\left(\,\{\,d_{\,i}\,\}\,\right)\,(\,x,\,a_{\,2},\, \cdots,\, a_{\,n}\,)\,\right|\]
\[\hspace{.5cm} \leq \|\,T\,\|\left\|\,x,\, a_{\,2},\, \cdots,\, a_{\,n}\,\right\|_{X}\left(\,\sum\limits_{\,i \,\in\, I}\,|\,d_{\,i}\,|^{\,q}\,\right)^{\,1 \,/\, q}.\]
Thus
\[\sum\limits_{\,i \,\in\, I}\,\left|\,T_{\,i}\,(\,x,\,a_{\,2},\,\cdots,\,a_{\,n}\,)\,\right|^{\,p} \,=\, \left\|\,\phi_{\,x}\,\right\|^{\,p} \,\leq\, \|\,T\,\|^{\,p}\,\left\|\,x,\, a_{\,2},\, \cdots,\, a_{\,n}\,\right\|_{X}^{\,p}\]
\[\hspace{2.4cm} \,\leq\, B\,\left\|\,x,\, a_{\,2},\, \cdots,\, a_{\,n}\,\right\|_{X}^{\,p}.\]
This shows that $\left\{\,T_{i}\,\right\}_{\,i \,\in\, I}$\, is \,$p$-Bessel sequence associated to \,$\left(\,a_{\,2},\, \cdots,\, a_{\,n}\,\right)$\, with bound \,$B$.\,This completes the proof.\\   
\end{proof}

Let \,$\left(\,Y,\, \|\,\cdot,\, \cdots,\, \cdot\,\|_{Y} \,\right)$\, be a \,$n$-Banach space.\,Then the Cartesian product of \,$X$\, and \,$Y$\, is denoted by \,$X \,\oplus\, Y$\, and defined to be an \,$n$-Banach space with respect to the \,$n$-norm 
\[\left\|\,x_{\,1} \,\oplus\, y_{\,1},\, x_{\,2} \,\oplus\, y_{\,2},\, \cdots,\, x_{\,n} \,\oplus\, y_{\,n}\,\right\|^{\,p} \,=\, \left\|\,x_{\,1},\, x_{\,2},\, \cdots,\, x_{\,n}\,\right\|^{\,p}_{X} \,+\, \left\|\,y_{\,1},\, y_{\,2},\, \cdots,\, y_{\,n}\,\right\|^{\,p}_{Y},\]
for all \,$x_{\,1} \,\oplus\, y_{\,1},\, x_{\,2} \,\oplus\, y_{\,2},\,\cdots,\, x_{\,n} \,\oplus\, y_{\,n} \,\in\, X \,\oplus\, Y$, and \,$x_{\,1},\, x_{\,2},\, \cdots,\, x_{\,n} \,\in\, X$; \,$\,y_{\,1},\, y_{\,2},\, \cdots,\, y_{\,n} \,\in\, Y$.\,According to the definition (\ref{def1}), consider \,$Y^{\,\ast}_{G}$\, as the Banach space of all bounded\;$b$-linear functional defined on \,$Y \,\times\, \left<\,b_{\,2}\,\right> \,\times \cdots \,\times\, \left<\,b_{\,n}\,\right>$\, and \,$Z^{\,\ast}$\, as the Banach space of all bounded\;$b$-linear functional defined on \,$X \,\oplus\, Y \,\times\, \left<\,a_{\,2} \,\oplus\, b_{\,2}\,\right> \,\times \cdots \,\times\, \left<\,a_{\,n} \,\oplus\, b_{\,n}\,\right>$, where \,$b_{\,2},\, \cdots,\,b_{\,n} \,\in\, Y$\, and \,$a_{\,2} \,\oplus\, b_{\,2},\, \cdots,\, a_{\,n} \,\oplus\, b_{\,n} \,\in\, X \,\oplus\, Y$\, are fixed elements.\,Now, if \,$T \,\in\, X_{F}^{\,\ast}$\, and \,$U \,\in\, Y^{\,\ast}_{G}$, for all \,$x \,\oplus\, y \,\in\, X \,\oplus\, Y$, we define \,$T \,\oplus\, U \,\in\, Z^{\,\ast}$\, by 
\[(\,T \,\oplus\, U\,)\,\left(\,x \,\oplus\, y,\, a_{\,2} \,\oplus\, b_{\,2},\, \cdots,\, a_{\,n} \,\oplus\, b_{\,n}\,\right)\hspace{3.3cm}\]
\[=\, \left(\,T\,(\,x,\, a_{\,2},\, \cdots,\, a_{\,n}\,),\, U\,(\,y,\, b_{\,2},\, \cdots,\, b_{\,n}\,)\,\right)\; \;\forall\; x \,\in\, X,\; y \,\in\, Y.\]

\begin{theorem}
Let \,$\left\{\,T_{\,i}\,\right\}_{i \,\in\, I} \,\subseteq\, X_{F}^{\,\ast}$\, be a \,$p$-frame associated to \,$\left(\,a_{\,2},\, \cdots,\, a_{\,n}\,\right)$\, for \,$X$\, with bounds \,$A,\,B$\, and \,$\left\{\,U_{\,i}\,\right\}_{i \,\in\, I} \,\subseteq\, Y_{G}^{\,\ast}$\, be a \,$p$-frame associated to \,$\left(\,b_{\,2},\, \cdots,\, b_{\,n}\,\right)$\, for \,$Y$\, with bounds \,$C,\,D$.\,Then the sequence \,$\left\{\,T_{\,i} \,\oplus\, U_{\,i}\,\right\}_{i \,\in\, I}$\, is a \,$p$-frame associated to \,$\left(\,a_{\,2} \,\oplus\, b_{\,2},\, \cdots,\, a_{\,n} \,\oplus\,  b_{\,n}\,\right)$\, for \,$X \,\oplus\, Y$\, with bounds \,$\min\,(\,A,C\,)$\, and \,$\max\,(\,B,\,D\,)$.  
\end{theorem}

\begin{proof}
Since \,$\left\{\,T_{\,i}\,\right\}_{i \,\in\, I}$\, is a \,$p$-frame associated to \,$\left(\,a_{\,2},\, \cdots,\, a_{\,n}\,\right)$\, for \,$X$\, and \,$\left\{\,U_{\,i}\,\right\}_{i \,\in\, I}$\, is a \,$p$-frame associated to \,$\left(\,b_{\,2},\, \cdots,\, b_{\,n}\,\right)$\, for \,$Y$, for all \,$x \,\in\, X$\, and \,$y \,\in\, Y$, we have
\begin{equation}\label{eq1.1}
A \,\left\|\,x,\, a_{\,2},\, \cdots,\, a_{\,n}\,\right\|^{\,p}_{X} \,\leq\, \sum\limits_{\,i \,\in\, I}\,\left|\,T_{\,i}\,(\,x,\, a_{\,2},\, \cdots,\, a_{\,n}\,)\,\right|^{\,p} \,\leq\, B\,\left\|\,x,\, a_{\,2},\, \cdots,\, a_{\,n}\,\right\|^{\,p}_{X}
\end{equation}
\begin{equation}\label{eq1.2}
C \,\left\|\,y,\, b_{\,2},\, \cdots,\, b_{\,n}\,\right\|^{\,p}_{Y} \,\leq\, \sum\limits_{\,i \,\in\, I}\,\left|\,U_{\,i}\,(\,y,\, b_{\,2},\, \cdots,\, b_{\,n}\,)\,\right|^{\,p} \,\leq\, D\,\left\|\,y,\, b_{\,2},\, \cdots,\, b_{\,n}\,\right\|^{\,p}_{Y}
\end{equation} 
Adding (\ref{eq1.1}) and (\ref{eq1.2}), we get
\[A \,\left\|\,x,\, a_{\,2},\, \cdots,\, a_{\,n}\,\right\|^{\,p}_{X} \,+\, C \,\left\|\,y,\, b_{\,2},\, \cdots,\, b_{\,n}\,\right\|^{\,p}_{Y}\hspace{4.4cm}\]
\[\leq\, \sum\limits_{\,i \,\in\, I}\,\left|\,T_{\,i}\,(\,x,\, a_{\,2},\, \cdots,\, a_{\,n}\,)\,\right|^{\,p} \,+\, \sum\limits_{\,i \,\in\, I}\,\left|\,U_{\,i}\,(\,y,\, b_{\,2},\, \cdots,\, b_{\,n}\,)\,\right|^{\,p}\hspace{2cm}\]
\[\leq\, B\,\left\|\,x,\, a_{\,2},\, \cdots,\, a_{\,n}\,\right\|^{\,p}_{X} \,+\, D\,\left\|\,y,\, b_{\,2},\, \cdots,\, b_{\,n}\,\right\|^{\,p}_{Y}.\hspace{3.8cm}\]
\[\Rightarrow\, \min\,(\,A,C\,)\,\left\{\,\left\|\,x,\, a_{\,2},\, \cdots,\, a_{\,n}\,\right\|^{\,p}_{X} \,+\, \left\|\,y,\, b_{\,2},\, \cdots,\, b_{\,n}\,\right\|^{\,p}_{Y}\,\right\}\hspace{2.1cm}\]
\[\leq\, \sum\limits_{\,i \,\in\, I}\,\left(\,\left|\,T_{\,i}\,(\,x,\, a_{\,2},\, \cdots,\, a_{\,n}\,)\,\right|^{\,p} \,+\, \left|\,U_{\,i}\,(\,y,\, b_{\,2},\, \cdots,\, b_{\,n}\,)\,\right|^{\,p}\,\right)\hspace{2.3cm}\]
\[\leq\, \max\,(\,B,\,D\,)\,\left\{\,\left\|\,x,\, a_{\,2},\, \cdots,\, a_{\,n}\,\right\|^{\,p}_{X} \,+\, \left\|\,y,\, b_{\,2},\, \cdots,\, b_{\,n}\,\right\|^{\,p}_{Y}\,\right\}.\hspace{1.8cm}\]
\[\Rightarrow\, \min\,(\,A,C\,)\,\left\|\,x \,\oplus\, y,\, a_{\,2} \,\oplus\, b_{\,2},\, \cdots,\, a_{\,n} \,\oplus\, b_{\,n}\,\right\|^{\,p}\hspace{3.6cm}\]
\[\,\leq\, \sum\limits_{\,i \,\in\, I}\,\left|\,(\,T_{\,i} \,\oplus\, U_{\,i}\,)\,\left(\,x \,\oplus\, y,\, a_{\,2} \,\oplus\, b_{\,2},\, \cdots,\, a_{\,n} \,\oplus\, b_{\,n}\,\right)\,\right|^{\,p}\hspace{2.8cm}\]
\[\leq\, \max\,(\,B,\,D\,)\,\left\|\,x \,\oplus\, y,\, a_{\,2} \,\oplus\, b_{\,2},\, \cdots,\, a_{\,n} \,\oplus\, b_{\,n}\,\right\|^{\,p}\; \;\forall\; x \,\oplus\, y \,\in\, X \,\oplus\, Y.\]
This shows that \,$\left\{\,T_{\,i} \,\oplus\, U_{\,i}\,\right\}_{i \,\in\, I}$\, is a \,$p$-frame associated to \,$\left(\,a_{\,2} \,\oplus\, b_{\,2},\, \cdots,\, a_{\,n} \,\oplus\, b_{\,n}\,\right)$\, for \,$X \,\oplus\, Y$\, with bounds \,$\min\,(\,A,C\,)$\, and \,$\max\,(\,B,\,D\,)$.
\end{proof}

\section{Perturbation of $p$-frame in \,$n$-Banach space}

\smallskip\hspace{.6 cm} In this section, perturbation of a \,$p$-frame associated to \,$\left(\,a_{\,2},\, \cdots,\, a_{\,n}\,\right)$\, for \,$X$\, by non-zero bounded\;$b$-linear functionals is presented.

\begin{theorem}
Let \,$\left\{\,T_{\,i}\,\right\}_{i \,\in\, I} \,\subseteq\, X_{F}^{\,\ast}$\, be a \,$p$-frame associated to \,$\left(\,a_{\,2},\, \cdots,\, a_{\,n}\,\right)$\, for \,$X$\, with bounds \,$A$\, and \,$B$.\,Let \,$R \,\neq\, 0$\, be any element in \,$X_{F}^{\,\ast}$\, and \,$\left\{\,c_{\,i}\,\right\}_{i \,\in\, I}$\, be any sequence of scalars.\,Then the perturbed sequence of bounded \,$b$-linear functionals \,$\left\{\,T_{\,i} \,+\, c_{\,i}\,R\,\right\}_{i \,\in\, I}$\, is a \,$p$-frame associated to \,$\left(\,a_{\,2},\, \cdots,\, a_{\,n}\,\right)$\, for \,$X$\, if \,$\sum\limits_{\,i \,\in\, I}\,\left|\,c_{\,i}\,\right|^{\,p} \,<\, \dfrac{A}{\|\,R\,\|^{\,p}}$.   
\end{theorem}

\begin{proof}
Let \,$U_{\,i} \,=\, T_{\,i} \,+\, c_{\,i}\,R,\, i \,\in\, I$.\,Then for each \,$x \,\in\, X$, we have
\[\sum\limits_{\,i \,\in\, I}\,\left|\,T_{\,i}\,(\,x,\, a_{\,2},\, \cdots,\, a_{\,n}\,) \,-\, U_{\,i}\,(\,x,\, a_{\,2},\, \cdots,\, a_{\,n}\,)\,\right|^{\,p}\hspace{4cm}\]
\[=\, \sum\limits_{\,i \,\in\, I}\,\left|\,c_{\,i}\,R\,(\,x,\, a_{\,2},\, \cdots,\, a_{\,n}\,)\,\right|^{\,p} \,\leq\, \sum\limits_{\,i \,\in\, I}\,\left|\,c_{\,i}\,\right|^{\,p}\,\|\,R\,\|^{\,p}\,\left\|\,x,\, a_{\,2},\, \cdots,\, a_{\,n}\,\right\|^{\,p}_{X}\]
\[=\, M\,\left\|\,x,\, a_{\,2},\, \cdots,\, a_{\,n}\,\right\|^{\,p}_{X},\, \,\text{where}\; \,M \,=\, \sum\limits_{\,i \,\in\, I}\,\left|\,c_{\,i}\,\right|^{\,p}\,\|\,R\,\|^{\,p}.\hspace{2.36cm}\]
Therefore, \,$\left\{\,T_{\,i} \,+\, c_{\,i}\,R\,\right\}_{i \,\in\, I}$\, is a \,$p$-frame associated to \,$\left(\,a_{\,2},\, \cdots,\, a_{\,n}\,\right)$\, if \,$M \,<\, A$, \,i\,.\,e., if \,$\sum\limits_{\,i \,\in\, I}\,\left|\,c_{\,i}\,\right|^{\,p}\,\|\,R\,\|^{\,p} \,<\, A\; \;\Rightarrow\, \sum\limits_{\,i \,\in\, I}\,\left|\,c_{\,i}\,\right|^{\,p} \,<\, \dfrac{A}{\|\,R\,\|^{\,p}}$. 
\end{proof}

\begin{theorem}
Let \,$\left\{\,T_{\,i}\,\right\}_{i \,\in\, I}$\, be a \,$p$-frame associated to \,$\left(\,a_{\,2},\, \cdots,\, a_{\,n}\,\right)$\, for \,$X$\, with bounds \,$A,\,B$\, and \,$\left\{\,R_{\,i}\,\right\}_{i \,\in\, I} \,\subseteq\, X_{F}^{\,\ast}$\, be any sequence.\,Let \,$\left\{\,\alpha_{\,i}\,\right\}_{i \,\in\, I},\;\left\{\,\beta_{\,i}\,\right\}_{i \,\in\, I} \,\subset\, \mathbb{R}$\, be any two positively confined sequence such that
\[\sum\limits_{\,i \,\in\, I}\,\left|\,\alpha_{\,i}\,T_{\,i}\,(\,x,\, a_{\,2},\, \cdots,\, a_{\,n}\,) \,-\, \beta_{\,i}\,R_{\,i}\,(\,x,\, a_{\,2},\, \cdots,\, a_{\,n}\,)\,\right|^{\,p}\]
\[\,\leq\, \lambda\,\sum\limits_{\,i \,\in\, I}\,\left|\,\alpha_{\,i}\,T_{\,i}\,(\,x,\, a_{\,2},\, \cdots,\, a_{\,n}\,)\,\right|^{\,p} \,+\, \mu\,\sum\limits_{\,i \,\in\, I}\,\left|\,\beta_{\,i}\,R_{\,i}\,(\,x,\, a_{\,2},\, \cdots,\, a_{\,n}\,)\,\right|^{\,p},\, \,x \,\in\, X,\]
where \,$\lambda,\,\mu$\, are constants with \,$0 \,\leq\, \lambda,\,\mu \,<\, \dfrac{1}{2^{\,p}}$.\,Then \,$\left\{\,R_{\,i}\,\right\}_{i \,\in\, I}$\, is a \,$p$-frame associated to \,$\left(\,a_{\,2},\, \cdots,\, a_{\,n}\,\right)$\, for \,$X$.   
\end{theorem}

\begin{proof} 
For each \,$x \,\in\, X$, we have
\[\sum\limits_{\,i \,\in\, I}\,\left|\,\beta_{\,i}\,R_{\,i}\,(\,x,\, a_{\,2},\, \cdots,\, a_{\,n}\,)\,\right|^{\,p} \,\leq\, 2^{\,p}\sum\limits_{\,i \,\in\, I}\,\left|\,\alpha_{\,i}\,T_{\,i}\,(\,x,\, a_{\,2},\, \cdots,\, a_{\,n}\,)\,\right|^{\,p} \,+\hspace{2.2cm}\]
\[+\, 2^{\,p}\sum\limits_{\,i \,\in\, I}\,\left|\,\alpha_{\,i}\,T_{\,i}\,(\,x,\, a_{\,2},\, \cdots,\, a_{\,n}\,) \,-\, \beta_{\,i}\,R_{\,i}\,(\,x,\, a_{\,2},\, \cdots,\, a_{\,n}\,)\,\right|^{\,p}\;[\;\text{by $(\ref{eqq1})$}\;]\]
\[\leq\, 2^{\,p}\,\left(\,1 \,+\, \lambda\,\right)\,\sum\limits_{\,i \,\in\, I}\,\left|\,\alpha_{\,i}\,T_{\,i}\,(\,x,\, a_{\,2},\, \cdots,\, a_{\,n}\,)\,\right|^{\,p} \,+\, 2^{\,p}\,\mu\,\sum\limits_{\,i \,\in\, I}\,\left|\,\beta_{\,i}\,R_{\,i}\,(\,x,\, a_{\,2},\, \cdots,\, a_{\,n}\,)\,\right|^{\,p}.\]
This implies that
\[\left(\,1 - 2^{\,p}\,\mu\,\right)M^{\,p}\sum\limits_{\,i \,\in\, I}\left|\,R_{\,i}\,(\,x,\, a_{\,2},\, \cdots,\, a_{\,n}\,)\,\right|^{\,p} \leq 2^{\,p}\,\left(\,1 + \lambda\,\right)N^{\,p}\sum\limits_{\,i \,\in\, I}\left|\,T_{\,i}\,(\,x,\, a_{\,2},\, \cdots,\, a_{\,n}\,)\,\right|^{\,p}\]
where \,$M \,=\, \inf\limits_{\,i}\,\beta_{\,i}$\, and \,$N \,=\, \sup\limits_{\,i}\,\alpha_{\,i}$\, and therefore we can write
\[\sum\limits_{\,i \,\in\, I}\,\left|\,R_{\,i}\,(\,x,\, a_{\,2},\, \cdots,\, a_{\,n}\,)\,\right|^{\,p} \,\leq\, \dfrac{2^{\,p}\,\left(\,1 + \lambda\,\right)\,N^{\,p}}{\left(\,1 - 2^{\,p}\,\mu\,\right)\,M^{\,p}}\,\sum\limits_{\,i \,\in\, I}\,\left|\,T_{\,i}\,(\,x,\, a_{\,2},\, \cdots,\, a_{\,n}\,)\,\right|^{\,p}\hspace{2cm}\]
\begin{equation}\label{eq1.3}
\hspace{3cm}=\, \dfrac{2^{\,p}\,\left(\,1 + \lambda\,\right)\,N^{\,p}}{\left(\,1 - 2^{\,p}\,\mu\,\right)\,M^{\,p}}\,B\,\left\|\,x,\, a_{\,2},\, \cdots,\, a_{\,n}\,\right\|^{\,p}_{X}\; \;\forall\; x \,\in\, X. 
\end{equation}
\[\bigg[\,\text{since \,$\left\{\,T_{\,i}\,\right\}_{i \,\in\, I}$\, is a \,$p$-frame associated to \,$\left(\,a_{\,2},\, \cdots,\, a_{\,n}\,\right)$}\,\bigg].\]
On the other hand, for each \,$x \,\in\, X$, we have 
\[\sum\limits_{\,i \,\in\, I}\,\left|\,\alpha_{\,i}\,T_{\,i}\,(\,x,\, a_{\,2},\, \cdots,\, a_{\,n}\,)\,\right|^{\,p}\hspace{9cm}\]
\[\leq\, 2^{\,p}\,\left(\,\sum\limits_{\,i \,\in\, I}\,\left|\,\left(\,\alpha_{\,i}\,T_{\,i} \,-\, \beta_{\,i}\,R_{\,i}\,\right)\,(\,x,\, a_{\,2},\, \cdots,\, a_{\,n}\,)\,\right|^{\,p} \,+\, \sum\limits_{\,i \,\in\, I}\,\left|\,\beta_{\,i}\,R_{\,i}\,(\,x,\, a_{\,2},\, \cdots,\, a_{\,n}\,)\,\right|^{\,p}\,\right)\]
\[\leq\, 2^{\,p}\,\lambda\,\sum\limits_{\,i \,\in\, I}\,\left|\,\alpha_{\,i}\,T_{\,i}\,(\,x,\, a_{\,2},\, \cdots,\, a_{\,n}\,)\,\right|^{\,p} \,+\, 2^{\,p}\,\left(\,1 + \mu\,\right)\,\sum\limits_{\,i \,\in\, I}\,\left|\,\beta_{\,i}\,R_{\,i}\,(\,x,\, a_{\,2},\, \cdots,\, a_{\,n}\,)\,\right|^{\,p}.\]Thus
\[\left(\,1 - 2^{\,p}\,\lambda\,\right)M_{1}^{\,p}\sum\limits_{\,i \,\in\, I}\left|\,T_{\,i}\,(\,x,\, a_{\,2},\, \cdots,\, a_{\,n}\,)\,\right|^{\,p} \leq 2^{\,p}\,\left(\,1 + \mu\,\right)N_{1}^{\,p}\sum\limits_{\,i \,\in\, I}\left|\,R_{\,i}\,(\,x,\, a_{\,2},\, \cdots,\, a_{\,n}\,)\,\right|^{\,p}\]
where \,$M_{1} \,=\, \inf\limits_{\,i}\,\alpha_{\,i}$\, and \,$N_{1} \,=\, \sup\limits_{\,i}\,\beta_{\,i}$.
\begin{equation}\label{eq1.4}
\Rightarrow\, \dfrac{\left(\,1 - 2^{\,p}\,\lambda\,\right)\,M_{1}^{\,p}}{2^{\,p}\left(\,1 + \mu\,\right)\,N_{1}^{\,p}}\,A\,\left\|\,x,\, a_{\,2},\, \cdots,\, a_{\,n}\,\right\|^{\,p}_{X} \,\leq\, \sum\limits_{\,i \,\in\, I}\,\left|\,R_{\,i}\,(\,x,\, a_{\,2},\, \cdots,\, a_{\,n}\,)\,\right|^{\,p}.
\end{equation}
From (\ref{eq1.3}) and (\ref{eq1.4}), it follows that \,$\left\{\,R_{\,i}\,\right\}_{i \,\in\, I}$\, is a \,$p$-frame associated to \,$\left(\,a_{\,2},\, \cdots,\, a_{\,n}\,\right)$\, for \,$X$. 
\end{proof}

\section{Stability of $p$-frame in \,$n$-Banach space}

In this section, we present the stability of \,$p$-frame associated to \,$\left(\,a_{\,2},\, \cdots,\, a_{\,n}\,\right)$\, in \,$n$-Banach space and prove that \,$p$-frames associated to \,$\left(\,a_{\,2},\, \cdots,\, a_{\,n}\,\right)$\, are stable under some perturbations.  

\begin{theorem}\label{th2.1}
Let \,$\left\{\,T_{\,i}\,\right\}_{i \,\in\, I}$\, be a \,$p$-frame associated to \,$\left(\,a_{\,2},\, \cdots,\, a_{\,n}\,\right)$\, for \,$X$\, with bounds \,$A,\,B$\, and \,$\left\{\,R_{\,i}\,\right\}_{i \,\in\, I} \,\subseteq\, X_{F}^{\,\ast}$\, be any sequence such that
\[\sum\limits_{\,i \,\in\, I}\,\left|\,T_{\,i}\,(\,x,\, a_{\,2},\, \cdots,\, a_{\,n}\,) \,-\, R_{\,i}\,(\,x,\, a_{\,2},\, \cdots,\, a_{\,n}\,)\,\right|^{\,p}\]
\[\,\leq\, \alpha\,\sum\limits_{\,i \,\in\, I}\,\left|\,T_{\,i}\,(\,x,\, a_{\,2},\, \cdots,\, a_{\,n}\,)\,\right|^{\,p} \,+\, \beta\,\left\|\,x,\, a_{\,2},\, \cdots,\, a_{\,n}\,\right\|^{\,p}_{X},\, \,x \,\in\, X,\]
where \,$\alpha,\,\beta \,\geq\, 0$\, with \,$0 \,\leq\, \alpha \,+\, \dfrac{\beta}{A} \,<\, 1$.\,Then \,$\left\{\,R_{\,i}\,\right\}_{i \,\in\, I}$\, is a \,$p$-frame associated to \,$\left(\,a_{\,2},\, \cdots,\, a_{\,n}\,\right)$\, for \,$X$. 
\end{theorem}

\begin{proof}
For each \,$x \,\in\, X$, by Minkowski inequality, we have 
\[\left(\,\sum\limits_{\,i \,\in\, I}\,\left|\,T_{\,i}\,(\,x,\, a_{\,2},\, \cdots,\, a_{\,n}\,)\,\right|^{\,p}\,\right)^{\,1 \,/\, p}\hspace{9cm}\]
\[\,\leq\, \left(\,\sum\limits_{\,i \,\in\, I}\,\left|\,\left(\,T_{\,i} \,-\, R_{\,i}\,\right)\,(\,x,\, a_{\,2},\, \cdots,\, a_{\,n}\,)\,\right|^{\,p}\,\right)^{\,1 \,/\, p} \,+\, \left(\,\sum\limits_{\,i \,\in\, I}\,\left|\,R_{\,i}\,(\,x,\, a_{\,2},\, \cdots,\, a_{\,n}\,)\,\right|^{\,p}\,\right)^{\,1 \,/\, p}\]
\[\leq \left(\alpha \sum\limits_{\,i \,\in\, I} \left|\,T_{\,i}\,(\,x,\, a_{\,2},\, \cdots,\, a_{\,n}\,)\,\right|^{\,p} \,+\, \beta\,\left\|\,x,\, a_{\,2},\, \cdots,\, a_{\,n}\,\right\|^{\,p}_{X}\right)^{\,1 \,/\, p} \,+\, \left(\sum\limits_{\,i \,\in\, I}\left|\,R_{\,i}\,(\,x,\, a_{\,2},\, \cdots,\, a_{\,n}\,)\,\right|^{\,p}\right)^{\,1 \,/\, p}\]
\[\leq\, \left(\,\left(\,\alpha \,+\, \dfrac{\beta}{A}\,\right)\,\sum\limits_{\,i \,\in\, I}\,\left|\,T_{\,i}\,(\,x,\, a_{\,2},\, \cdots,\, a_{\,n}\,)\,\right|^{\,p}\,\right)^{\,1 \,/\, p} \,+\, \left(\,\sum\limits_{\,i \,\in\, I}\,\left|\,R_{\,i}\,(\,x,\, a_{\,2},\, \cdots,\, a_{\,n}\,)\,\right|^{\,p}\,\right)^{\,1 \,/\, p}\]
\[\Rightarrow\, \left(\,1 \,-\, \left(\,\alpha \,+\, \dfrac{\beta}{A}\,\right)^{\,1 \,/\, p}\,\right)^{\,p}\,\sum\limits_{\,i \,\in\, I}\,\left|\,T_{\,i}\,(\,x,\, a_{\,2},\, \cdots,\, a_{\,n}\,)\,\right|^{\,p} \,\leq\, \sum\limits_{\,i \,\in\, I}\,\left|\,R_{\,i}\,(\,x,\, a_{\,2},\, \cdots,\, a_{\,n}\,)\,\right|^{\,p}\]
\[\Rightarrow\,\left(\,1 \,-\, \left(\,\alpha \,+\, \dfrac{\beta}{A}\,\right)^{\,1 \,/\, p}\,\right)^{\,p}\,A\,\left\|\,x,\, a_{\,2},\, \cdots,\, a_{\,n}\,\right\|^{\,p}_{X} \,\leq\, \sum\limits_{\,i \,\in\, I}\,\left|\,R_{\,i}\,(\,x,\, a_{\,2},\, \cdots,\, a_{\,n}\,)\,\right|^{\,p}.\]
On the other hand, for each \,$x\,\in\, X$, we have 
\[\left(\,\sum\limits_{\,i \,\in\, I}\,\left|\,R_{\,i}\,(\,x,\, a_{\,2},\, \cdots,\, a_{\,n}\,)\,\right|^{\,p}\,\right)^{\,1 \,/\, p}\hspace{9cm}\]
\[\,\leq\, \left(\,\sum\limits_{\,i \,\in\, I}\,\left|\,\left(\,T_{\,i} \,-\, R_{\,i}\,\right)\,(\,x,\, a_{\,2},\, \cdots,\, a_{\,n}\,)\,\right|^{\,p}\,\right)^{\,1 \,/\, p} \,+\, \left(\,\sum\limits_{\,i \,\in\, I}\,\left|\,T_{\,i}\,(\,x,\, a_{\,2},\, \cdots,\, a_{\,n}\,)\,\right|^{\,p}\,\right)^{\,1 \,/\, p}\]
\[\,\leq\, \left(\,\alpha\,\sum\limits_{\,i \,\in\, I}\,\left|\,T_{\,i}\,(\,x,\, a_{\,2},\, \cdots,\, a_{\,n}\,)\,\right|^{\,p} \,+\, \beta\,\left\|\,x,\, a_{\,2},\, \cdots,\, a_{\,n}\,\right\|^{\,p}_{X}\,\right)^{\,1 \,/\, p} \,+\, B^{\,1 \,/\, p}\,\left\|\,x,\, a_{\,2},\, \cdots,\, a_{\,n}\,\right\|_{X}\]
\[\leq\, \left(\,\left(\,\alpha\,\,B \,+\, \beta\,\right)^{\,1 \,/\, p} \,+\, B^{\,1 \,/\, p}\,\right)\,\left\|\,x,\, a_{\,2},\, \cdots,\, a_{\,n}\,\right\|_{X}\; \;\forall\; x\,\in\, X.\hspace{3.45cm}\]
\[\Rightarrow \sum\limits_{\,i \,\in\, I}\left|\,R_{\,i}\,(\,x,\, a_{\,2},\, \cdots,\, a_{\,n}\,)\,\right|^{\,p} \leq \left(\,\left(\,\alpha\,\,B \,+\, \beta\,\right)^{\,1 \,/\, p} \,+\, B^{\,1 \,/\, p}\,\right)^{\,p}\,\left\|\,x,\, a_{\,2},\, \cdots,\, a_{\,n}\,\right\|^{\,p}_{X}.\]
Therefore, \,$\left\{\,R_{\,i}\,\right\}_{i \,\in\, I}$\, is a \,$p$-frame associated to \,$\left(\,a_{\,2},\, \cdots,\, a_{\,n}\,\right)$\, for \,$X$. 
\end{proof}

\begin{corollary}
Let \,$\left\{\,T_{\,i}\,\right\}_{i \,\in\, I}$\, be a \,$p$-frame associated to \,$\left(\,a_{\,2},\, \cdots,\, a_{\,n}\,\right)$\, for \,$X$\, with bounds \,$A,\,B$\, and \,$\left\{\,R_{\,i}\,\right\}_{i \,\in\, I} \,\subseteq\, X_{F}^{\,\ast}$\, be any sequence such that
\[\sum\limits_{\,i \,\in\, I}\,\left|\,\left(\,T_{\,i} \,-\, R_{\,i}\,\right)\,(\,x,\, a_{\,2},\, \cdots,\, a_{\,n}\,)\,\right|^{\,p} \,\leq\, R\,\left\|\,x,\, a_{\,2},\, \cdots,\, a_{\,n}\,\right\|^{\,p}_{X}\]
for all \,$\,x \,\in\, X$, where \,$0 \,<\, R \,<\, A$.\,Then \,$\left\{\,R_{\,i}\,\right\}_{i \,\in\, I}$\, is a \,$p$-frame associated to \,$\left(\,a_{\,2},\, \cdots,\, a_{\,n}\,\right)$\, for \,$X$.
\end{corollary}

\begin{proof}
The proof of this corollary is directly follows from the Theorem (\ref{th2.1}), by putting \,$\alpha \,=\, 0$.
\end{proof}

Now, we present a necessary and sufficient condition for the stability of a \,$p$-frame in \,$n$-Banach space. 

\begin{theorem}
Let \,$\left\{\,T_{\,i}\,\right\}_{i \,\in\, I}$\, be a \,$p$-frame associated to \,$\left(\,a_{\,2},\, \cdots,\, a_{\,n}\,\right)$\, for \,$X$\, with bounds \,$A,\,B$.\,Then a sequence \,$\left\{\,R_{\,i}\,\right\}_{i \,\in\, I} \,\subseteq\, X_{F}^{\,\ast}$\, is a \,$p$-frame associated to \,$\left(\,a_{\,2},\, \cdots,\, a_{\,n}\,\right)$\, for \,$X$\, if and only if there exists a constant \,$M \,>\, 0$\, such that
\[\sum\limits_{\,i \,\in\, I}\,\left|\,T_{\,i}\,(\,x,\, a_{\,2},\, \cdots,\, a_{\,n}\,) \,-\, R_{\,i}\,(\,x,\, a_{\,2},\, \cdots,\, a_{\,n}\,)\,\right|^{\,p}\]
\[\,\leq\, M\,\min\left\{\,\sum\limits_{\,i \,\in\, I}\,\left|\,T_{\,i}\,(\,x,\, a_{\,2},\, \cdots,\, a_{\,n}\,)\,\right|^{\,p},\, \sum\limits_{\,i \,\in\, I}\,\left|\,R_{\,i}\,(\,x,\, a_{\,2},\, \cdots,\, a_{\,n}\,)\,\right|^{\,p}\,\right\}\; \;\forall\, x \,\in\, X.\]  
\end{theorem}

\begin{proof}
Since \,$\left\{\,T_{\,i}\,\right\}_{i \,\in\, I}$\, is a \,$p$-frame associated to \,$\left(\,a_{\,2},\, \cdots,\, a_{\,n}\,\right)$\, for \,$X$\, with bounds \,$A,\,B$, for each \,$x \,\in\, X$, we have
\[A\,\left\|\,x,\, a_{\,2},\, \cdots,\, a_{\,n}\,\right\|^{\,p}_{X}  \,\leq\, \sum\limits_{\,i \,\in\, I}\,\left|\,T_{\,i}\,(\,x,\, a_{\,2},\, \cdots,\, a_{\,n}\,)\,\right|^{\,p}\]
\[\leq\, \left[\,\left(\,\sum\limits_{\,i \,\in\, I}\,\left|\,\left(\,T_{\,i} \,-\, R_{\,i}\,\right)\,(\,x,\, a_{\,2},\, \cdots,\, a_{\,n}\,)\,\right|^{\,p}\,\right)^{\,1 \,/\, p} \,+\, \left(\,\sum\limits_{\,i \,\in\, I}\,\left|\,R_{\,i}\,(\,x,\, a_{\,2},\, \cdots,\, a_{\,n}\,)\,\right|^{\,p}\,\right)^{\,1 \,/\, p}\,\right]^{\,p}\]
\[\leq\, \left(\,M^{\,1 \,/\, p} \,+\, 1\,\right)^{\,p}\,\sum\limits_{\,i \,\in\, I}\,\left|\,R_{\,i}\,(\,x,\, a_{\,2},\, \cdots,\, a_{\,n}\,)\,\right|^{\,p}\hspace{7cm}\]
\[\Rightarrow\, \dfrac{A}{\,\left(\,M^{\,1 \,/\, p} \,+\, 1\,\right)^{\,p}}\,\left\|\,x,\, a_{\,2},\, \cdots,\, a_{\,n}\,\right\|^{\,p}_{X} \,\leq\, \sum\limits_{\,i \,\in\, I}\,\left|\,R_{\,i}\,(\,x,\, a_{\,2},\, \cdots,\, a_{\,n}\,)\,\right|^{\,p}.\hspace{2cm}\]
On the other hand, for each \,$x\,\in\, X$, we have \,$\sum\limits_{\,i \,\in\, I}\,\left|\,R_{\,i}\,(\,x,\, a_{\,2},\, \cdots,\, a_{\,n}\,)\,\right|^{\,p}$
\[\leq\, \left[\,\left(\,\sum\limits_{\,i \,\in\, I}\,\left|\,\left(\,T_{\,i} \,-\, R_{\,i}\,\right)\,(\,x,\, a_{\,2},\, \cdots,\, a_{\,n}\,)\,\right|^{\,p}\,\right)^{\,1 \,/\, p} \,+\, \left(\,\sum\limits_{\,i \,\in\, I}\,\left|\,T_{\,i}\,(\,x,\, a_{\,2},\, \cdots,\, a_{\,n}\,)\,\right|^{\,p}\,\right)^{\,1 \,/\, p}\,\right]^{\,p}\]
\[\leq\, \left(\,M^{\,1 \,/\, p} \,+\, 1\,\right)^{\,p}\,\sum\limits_{\,i \,\in\, I}\,\left|\,T_{\,i}\,(\,x,\, a_{\,2},\, \cdots,\, a_{\,n}\,)\,\right|^{\,p} \leq\, \left(\,M^{\,1 \,/\, p} \,+\, 1\,\right)^{\,p}\,B\,\left\|\,x,\, a_{\,2},\, \cdots,\, a_{\,n}\,\right\|^{\,p}_{X}.\]
Thus, \,$\left\{\,R_{\,i}\,\right\}_{i \,\in\, I}$\, is a \,$p$-frame associated to \,$\left(\,a_{\,2},\, \cdots,\, a_{\,n}\,\right)$\, for \,$X$.\\

Conversely, suppose that \,$\left\{\,R_{\,i}\,\right\}_{i \,\in\, I}$\, is a \,$p$-frame associated to \,$\left(\,a_{\,2},\, \cdots,\, a_{\,n}\,\right)$\, for \,$X$\, with bounds \,$C$\, and \,$D$.\,Then 
\[\sum\limits_{\,i \,\in\, I}\,\left|\,T_{\,i}\,(\,x,\, a_{\,2},\, \cdots,\, a_{\,n}\,) \,-\, R_{\,i}\,(\,x,\, a_{\,2},\, \cdots,\, a_{\,n}\,)\,\right|^{\,p}\hspace{5.5cm}\]
\[\leq\, \left[\,\left(\,\sum\limits_{\,i \,\in\, I}\,\left|\,T_{\,i}\,(\,x,\, a_{\,2},\, \cdots,\, a_{\,n}\,)\,\right|^{\,p}\,\right)^{\,1 \,/\, p} \,+\,  \left(\,\sum\limits_{\,i \,\in\, I}\,\left|\,R_{\,i}\,(\,x,\, a_{\,2},\, \cdots,\, a_{\,n}\,)\,\right|^{\,p}\,\right)^{\,1 \,/\, p}\,\right]^{\,p}\hspace{.5cm}\]
\[\leq\, \left[\,\left(\,\sum\limits_{\,i \,\in\, I}\,\left|\,T_{\,i}\,(\,x,\, a_{\,2},\, \cdots,\, a_{\,n}\,)\,\right|^{\,p}\,\right)^{\,1 \,/\, p} \,+\,  \left(\,D\,\left\|\,x,\, a_{\,2},\, \cdots,\, a_{\,n}\,\right\|^{\,p}_{X}\,\right)^{\,1 \,/\, p}\,\right]^{\,p}\hspace{2cm}\]
\[\leq\, \left[\,\left(\,\sum\limits_{\,i \,\in\, I}\,\left|\,T_{\,i}\,(\,x,\, a_{\,2},\, \cdots,\, a_{\,n}\,)\,\right|^{\,p}\,\right)^{\,1 \,/\, p} \,+\, \left(\,\dfrac{D}{A}\,\sum\limits_{\,i \,\in\, I}\,\left|\,T_{\,i}\,(\,x,\, a_{\,2},\, \cdots,\, a_{\,n}\,)\,\right|^{\,p}\,\right)^{\,1 \,/\, p}\,\right]^{\,p}\]
\[=\, \left(\,1 \,+\, \left(\,\dfrac{D}{A}\,\right)^{\,1 \,/\, p}\,\right)^{\,p}\,\sum\limits_{\,i \,\in\, I}\,\left|\,T_{\,i}\,(\,x,\, a_{\,2},\, \cdots,\, a_{\,n}\,)\,\right|^{\,p}\; \;\forall\; x \,\in\, X.\hspace{3.5cm}\]
Similarly, for each \,$x \,\in\, X$, we obtain
\[\sum\limits_{\,i \,\in\, I}\,\left|\,\left(\,T_{\,i} \,-\, R_{\,i}\,\right)\,(\,x,\, a_{\,2},\, \cdots,\, a_{\,n}\,)\,\right|^{\,p} \,\leq\, \left(\,1 \,+\, \left(\,\dfrac{B}{C}\,\right)^{\,1 \,/\, p}\,\right)^{p}\,\sum\limits_{\,i \,\in\, I}\,\left|\,R_{\,i}\,(\,x,\, a_{\,2},\, \cdots,\, a_{\,n}\,)\,\right|^{\,p}.\]
Hence, 
\[\sum\limits_{\,i \,\in\, I}\,\left|\,T_{\,i}\,(\,x,\, a_{\,2},\, \cdots,\, a_{\,n}\,) \,-\, R_{\,i}\,(\,x,\, a_{\,2},\, \cdots,\, a_{\,n}\,)\,\right|^{\,p}\hspace{3.3cm}\] 
\[\,\leq\, M\,\min\left\{\,\sum\limits_{\,i \,\in\, I}\,\left|\,T_{\,i}\,(\,x,\, a_{\,2},\, \cdots,\, a_{\,n}\,)\,\right|^{\,p},\, \sum\limits_{\,i \,\in\, I}\,\left|\,R_{\,i}\,(\,x,\, a_{\,2},\, \cdots,\, a_{\,n}\,)\,\right|^{\,p}\,\right\},\]
where \,$M \,=\, \min\left\{\left(\,1 \,+\, \left(\,\dfrac{D}{A}\,\right)^{\,1 \,/\, p}\,\right)^{\,p},\, \left(\,1 \,+\, \left(\,\dfrac{B}{C}\,\right)^{\,1 \,/\, p}\,\right)^{\,p}\right\}$.\,This completes the proof.    
\end{proof}

Next, we give a sufficient condition for the finite sum of \,$p$-frame associated to \,$\left(\,a_{\,2},\, \cdots,\, a_{\,n}\,\right)$\, to be a \,$p$-frame associated to \,$\left(\,a_{\,2},\, \cdots,\, a_{\,n}\,\right)$\, for \,$X$.

\begin{theorem}
For \,$k \,=\, 1,\,2,\,\cdots,\,l$, let \,$\left\{\,T_{\,k,\,i}\,\right\}_{i \,\in\, I} \,\subseteq\, X_{F}^{\,\ast}$\, be a \,$p$-frame associated to \,$\left(\,a_{\,2},\, \cdots,\, a_{\,n}\,\right)$\, for \,$X$\, and \,$\left\{\,\alpha_{\,k}\,\right\}_{k \,=\, 1}^{\,l}$\, be any scalars.\,Then \,$\left\{\,\sum\limits_{k \,=\, 1}^{\,l}\,\alpha_{\,k}\,T_{\,k,\,i}\,\right\}_{i \,\in\, I}$\, is a \,$p$-frame associated to \,$\left(\,a_{\,2},\, \cdots,\, a_{\,n}\,\right)$\, for \,$X$, if there exists \,$\beta \,>\, 0$\, and some \,$m \,\in\, \{\,1,\,2,\,\cdots,\,l\,\}$\, such that
\begin{equation}\label{eq1.5}
\beta\,\sum\limits_{\,i \,\in\, I}\,\left|\,T_{\,m,\,i}\,(\,x,\, a_{\,2},\, \cdots,\, a_{\,n}\,)\,\right|^{\,p} \,\leq\, \sum\limits_{\,i \,\in\, I}\,\left|\,\sum\limits_{k \,=\, 1}^{\,l}\,\alpha_{\,k}\,T_{\,k,\,i}\,(\,x,\, a_{\,2},\, \cdots,\, a_{\,n}\,)\,\right|^{\,p}. 
\end{equation}   
\end{theorem}  

\begin{proof}
For each \,$1 \,\leq\, m \,\leq\, l$, let \,$A_{\,m}$\, and \,$B_{\,m}$\, be the frame bounds of the \,$p$-frame \,$\left\{\,T_{\,m,\,i}\,\right\}_{i \,\in\, I}$\, associated to \,$\left(\,a_{\,2},\, \cdots,\, a_{\,n}\,\right)$\, for \,$X$.\,Let \,$\beta \,>\, 0$\, be a constant satisfying (\ref{eq1.5}).\,Then for each \,$x \,\in\, X$, we have
\[A_{\,m}\,\beta\,\left\|\,x,\, a_{\,2},\, \cdots,\, a_{\,n}\,\right\|^{\,p}_{X} \,\leq\, \beta\,\sum\limits_{\,i \,\in\, I}\,\left|\,T_{\,m,\,i}\,(\,x,\, a_{\,2},\, \cdots,\, a_{\,n}\,)\,\right|^{\,p}\]
\[\leq\, \sum\limits_{\,i \,\in\, I}\,\left|\,\sum\limits_{k \,=\, 1}^{\,l}\,\alpha_{\,k}\,T_{\,k,\,i}\,(\,x,\, a_{\,2},\, \cdots,\, a_{\,n}\,)\,\right|^{\,p}\; \;[\;\text{by (\ref{eq1.5})}\;].\]
On the other hand, for each \,$x \,\in\, X$, we have
\[\sum\limits_{\,i \,\in\, I}\,\left|\,\sum\limits_{k \,=\, 1}^{\,l}\,\alpha_{\,k}\,T_{\,k,\,i}\,(\,x,\, a_{\,2},\, \cdots,\, a_{\,n}\,)\,\right|^{\,p}\hspace{3.5cm}\]
\[ \,\leq\, \max\limits_{1 \,\leq\, k \,\leq\, l}\,\left|\,\alpha_{\,k}\,\right|^{\,p}\,\sum\limits_{k \,=\, 1}^{\,l}\,\left(\,\sum\limits_{\,i \,\in\, I}\,\left|\,T_{\,k,\,i}\,(\,x,\, a_{\,2},\, \cdots,\, a_{\,n}\,)\,\right|^{\,p}\,\right)\]
\[\,\leq\, \max\limits_{1 \,\leq\, k \,\leq\, l}\,\left|\,\alpha_{\,k}\,\right|^{\,p}\,\sum\limits_{k \,=\, 1}^{\,l}\,B_{\,k}\,\left\|\,x,\, a_{\,2},\, \cdots,\, a_{\,n}\,\right\|^{\,p}_{X},\hspace{1.5cm}\]
where \,$B_{\,k}$\, is the upper bound of the \,$p$-frame \,$\left\{\,T_{\,k,\,i}\,\right\}_{i \,\in\, I}$\, associated to \,$\left(\,a_{\,2},\, \cdots,\, a_{\,n}\,\right)$.
Hence, \,$\left\{\,\sum\limits_{k \,=\, 1}^{\,l}\,\alpha_{\,k}\,T_{\,k,\,i}\,\right\}_{i \,\in\, I}$\, is a \,$p$-frame associated to \,$\left(\,a_{\,2},\, \cdots,\, a_{\,n}\,\right)$\, for \,$X$.   
\end{proof}

\begin{theorem}
For \,$k \,=\, 1,\,2,\,\cdots,\,l$, let \,$\left\{\,T_{\,k,\,i}\,\right\}_{i \,\in\, I} \,\subseteq\, X_{F}^{\,\ast}$\, be a \,$p$-frame associated to \,$\left(\,a_{\,2},\, \cdots,\, a_{\,n}\,\right)$\, for \,$X$\, and \,$\left\{\,R_{\,k,\,i}\,\right\}_{i \,\in\, I} \,\subseteq\, X_{F}^{\,\ast}$\, be any sequence.\,Suppose \,$Q \,:\, l^{\,p} \,\to\, l^{\,p}$\, be a bounded linear operator  such that
\[Q\,\left(\,\left\{\,\sum\limits_{k \,=\, 1}^{\,l}\,R_{\,k,\,i}\,(\,x,\, a_{\,2},\, \cdots,\, a_{\,n}\,)\,\right\}\,\right) \,=\, \left\{\,T_{\,m,\,i}\,(\,x,\, a_{\,2},\, \cdots,\, a_{\,n}\,)\,\right\},\]
for some \,$m \,\in\, \{\,1,\,2,\,\cdots,\,l\,\}$.\,If there exists a non-negative constant \,$\lambda$\, such that
\[\sum\limits_{\,i \,\in\, I}\,\left|\,\left(\,T_{\,k,\,i} \,-\, R_{\,k,\,i}\,\right)\,(\,x,\, a_{\,2},\, \cdots,\, a_{\,n}\,)\,\right|^{\,p}\,\leq\, \lambda\,\sum\limits_{\,i \,\in\, I}\,\left|\,T_{\,k,\,i}\,(\,x,\, a_{\,2},\, \cdots,\, a_{\,n}\,)\,\right|^{\,p},\, \,x \,\in\, X,\]
\,$k \,=\, 1,\,2,\,\cdots,\,l$.\,Then \,$\left\{\,\sum\limits_{k \,=\, 1}^{\,l}\,R_{\,k,\,i}\,\right\}_{i \,\in\, I}$\, is a \,$p$-frame associated to \,$\left(\,a_{\,2},\, \cdots,\, a_{\,n}\,\right)$\, for \,$X$. 
\end{theorem}

\begin{proof}
For each \,$x \,\in\, X$, \,$\sum\limits_{\,i \,\in\, I}\,\left|\,\sum\limits_{k \,=\, 1}^{\,l}\,R_{\,k,\,i}\,(\,x,\, a_{\,2},\, \cdots,\, a_{\,n}\,)\,\right|^{\,p}$
\[\leq\, \sum\limits_{\,i \,\in\, I}\,\left[\,\left(\,\sum\limits_{k \,=\, 1}^{\,l}\,\left|\,\left(\,T_{\,k,\,i} \,-\, R_{\,k,\,i}\,\right)\,(\,x,\, a_{\,2},\, \cdots,\, a_{\,n}\,)\,\right|^{\,p}\,\right)^{\,1 \,/\, p} \,+\, \left(\,\sum\limits_{k \,=\, 1}^{\,l}\,\left|\,T_{\,k,\,i}\,(\,x,\, a_{\,2},\, \cdots,\, a_{\,n}\,)\,\right|^{\,p}\,\right)^{\,1 \,/\, p}\,\right]^{\,p}\]
\[\leq\, \sum\limits_{k \,=\, 1}^{\,l}\,\left(\,1 \,+\, \lambda^{\,1 \,/\, p}\,\right)^{\,p}\,\sum\limits_{\,i \,\in\, I}\,\left|\,T_{\,k,\,i}\,(\,x,\, a_{\,2},\, \cdots,\, a_{\,n}\,)\,\right|^{\,p}\leq\, \left(\,1 \,+\, \lambda^{\,1 \,/\, p}\,\right)^{\,p}\,\sum\limits_{k \,=\, 1}^{\,l}\,B_{\,k}\,\left\|\,x,\, a_{\,2},\, \cdots,\, a_{\,n}\,\right\|^{\,p}_{X},\]
where \,$B_{\,k}$\, is the upper bound of the \,$p$-frame \,$\left\{\,T_{\,k,\,i}\,\right\}_{i \,\in\, I}$\, associated to \,$\left(\,a_{\,2},\, \cdots,\, a_{\,n}\,\right)$.
Also, for each \,$x \,\in\, X$, we have
\[\left|\,Q\,\left(\,\left\{\,\sum\limits_{k \,=\, 1}^{\,l}\,R_{\,k,\,i}\,(\,x,\, a_{\,2},\, \cdots,\, a_{\,n}\,)\,\right\}\,\right)\,\right|^{\,p} \,=\, \sum\limits_{\,i \,\in\, I}\,\left|\,T_{\,m,\,i}\,(\,x,\, a_{\,2},\, \cdots,\, a_{\,n}\,)\,\right|^{\,p}.\]
Therefore, for each \,$x \,\in\, X$, we have
\[A_{\,m}\,\left\|\,x,\, a_{\,2},\, \cdots,\, a_{\,n}\,\right\|^{\,p}_{X} \,\leq\, \sum\limits_{\,i \,\in\, I}\,\left|\,T_{\,m,\,i}\,(\,x,\, a_{\,2},\, \cdots,\, a_{\,n}\,)\,\right|^{\,p}\hspace{4cm}\]
\[\hspace{2cm}\leq\, \|\,Q\,\|^{\,p}\,\sum\limits_{\,i \,\in\, I}\,\left|\,\sum\limits_{k \,=\, 1}^{\,l}\,R_{\,k,\,i}\,(\,x,\, a_{\,2},\, \cdots,\, a_{\,n}\,)\,\right|^{\,p},\]
where \,$A_{\,m}$\, is the lower bound of the \,$p$-frame \,$\left\{\,T_{\,m,\,i}\,\right\}_{i \,\in\, I}$\, associated to \,$\left(\,a_{\,2},\, \cdots,\, a_{\,n}\,\right)$.
\[\text{Thus,}\hspace{.4cm}\dfrac{A_{\,m}}{\|\,Q\,\|^{\,p}}\,\left\|\,x,\, a_{\,2},\, \cdots,\, a_{\,n}\,\right\|^{\,p}_{X} \,\leq\, \sum\limits_{\,i \,\in\, I}\,\left|\,\sum\limits_{k \,=\, 1}^{\,l}\,R_{\,k,\,i}\,(\,x,\, a_{\,2},\, \cdots,\, a_{\,n}\,)\,\right|^{\,p}.\hspace{3cm}\]
Hence, \,$\left\{\,\sum\limits_{k \,=\, 1}^{\,l}\,R_{\,k,\,i}\,\right\}_{i \,\in\, I}$\, is a \,$p$-frame associated to \,$\left(\,a_{\,2},\, \cdots,\, a_{\,n}\,\right)$\, for \,$X$.    
\end{proof}

\subsection{Compliance with Ethical Standards:}

\smallskip\hspace{.6 cm}{\bf Fund:} There are no funding sources.\\

{\bf Conflict of Interest:} First Author declares that he has no conflict of interest.\,Second Author declares that he has no conflict of interest.\\

{\bf Ethical approval:} This article does not contain any studies with human participants performed by any of the authors.

\end{document}